\documentclass[12pt]{amsart}
\usepackage{amssymb,amsmath}
\newtheorem{proposition}{Proposition}
\newtheorem{theorem}[proposition]{Theorem}

\newtheorem{lemma}[proposition]{Lemma}
\theoremstyle{definition}

\newtheorem{remark}[proposition]{Remark}
\title[A certain family of integrals]
{A direct computation of\\ a certain family of integrals}
\author[L. Fornari, E. Laeng and V. Pata]{Lorenzo fornari, Enrico Laeng and Vittorino Pata}

\address{Politecnico di Milano - Dipartimento di Matematica
\newline\indent
Via Bonardi 9, 20133 Milano, Italy}
\email{lorenzo.fornari@polimi.it {\rm (L. Fornari)}}
\email{enrico.laeng@polimi.it {\rm (E. Laeng)}}
\email{vittorino.pata@polimi.it {\rm (V. Pata)}}

\begin{document}

\begin{abstract}
We propose a rather elementary method to compute
a certain family of integrals on the half line,
depending on the integer parameters $n\geq q\geq 1$.
\end{abstract}

\maketitle


\noindent
In this note, let $n\geq q\geq 1$ be any two given integers.
The symbol $\lfloor\cdot\rfloor$ will stand, as usual, for the integer part.
We consider the family of integrals
$$I_{n,q}=\int_{0}^\infty\frac{(\sin x)^n}{x^q}\, dx.$$

\begin{theorem}
\label{UNO}
The following formulae hold:
\begin{itemize}
\item[(i)] If $n+q$ is even, then
$$
I_{n,q}
=\frac{(-1)^{\frac{q-n}2}\pi}{2^{n}(q-1)!}\sum_{k=0}^{\lfloor\frac{n-1}2\rfloor}(-1)^k
{\textstyle \bigg(\begin{matrix}
n \\
k\end{matrix}\bigg)}(n-2k)^{q-1}.
$$
\item[(ii)] If $n+q$ is odd and $q\geq 2$, then
$$
I_{n,q}
={\displaystyle \frac{(-1)^{\frac{q-n+1}2}}{2^{n-1}(q-1)!}\sum_{k=0}^{\lfloor\frac{n-1}2\rfloor}}(-1)^k
{\textstyle \bigg(\begin{matrix}
n \\
k\end{matrix}\bigg)}(n-2k)^{q-1}\log(n-2k).
$$
\end{itemize}
\end{theorem}

The formulae above are recorded in the Wolfram MathWorld web page titled
\emph{Sinc Function.}\footnote{{\tt https://mathworld.wolfram.com/SincFunction.html}}, which declares the result
``amazing" and ``spectacular".
However, the web page omits the proof, citing a twenty-year-old online paper of S.\ Kogan that seems not to
be available any longer. Nor the proof is reported in any textbook, to the best of our knowledge.

\smallskip
The remaining of the paper is devoted to our proof of Theorem~\ref{UNO}.
To this end, for $m\geq 0$, let
$$P_{m}(x)=\sum_{k=0}^{m}\frac{x^k}{k!}$$
denote the Maclaurin polynomial of $e^x$ of order $m$. We agree to set $P_{-1}=0$.
Let $Q(x)$ be the Maclaurin polynomial of $(\sin x)^n$ of order $q-2$, with
$Q=0$ if $q=1$. Since $(\sin x)^n$
has a zero of order $n$ at $x=0$, it follows that $Q(x)\equiv 0$ for all $n\geq q\geq 1$.
On the other hand, as
$$(\sin x)^n=\frac1{(2i)^n}\sum_{k=0}^{n}(-1)^k\left(\begin{matrix}
n \\
k\end{matrix}\right)e^{i(n-2k)x},
$$
we immediately conclude that
\begin{equation}
\label{QQQ}
Q(x)=\frac1{(2i)^n}\sum_{k=0}^{n}(-1)^k\left(\begin{matrix}
n \\
k\end{matrix}\right)P_{q-2}(i(n-2k)x)=0.
\end{equation}
Subtracting the two sums, we obtain
\begin{align}
\label{form}
I_{n,q}
&=\frac1{(2i)^n}\sum_{k=0}^{\lfloor\frac{n-1}2\rfloor}(-1)^k\left(\begin{matrix}
n \\
k\end{matrix}\right)\int_{0}^\infty\frac{e^{i(n-2k)x}-P_{q-2}(i(n-2k)x)}{x^q}\, dx\\
&\quad +\frac{(-1)^n}{(2i)^n}\sum_{k=0}^{\lfloor\frac{n-1}2\rfloor}(-1)^k\left(\begin{matrix}
n \\
k\end{matrix}\right)\int_{0}^\infty\frac{e^{-i(n-2k)x}-P_{q-2}(-i(n-2k)x)}{x^q}\, dx.\notag
\end{align}

\begin{remark} From \eqref{QQQ}, we also deduce that
the equality
\begin{equation}
\label{shortino}
\sum_{k=0}^{\lfloor\frac{n-1}2\rfloor}(-1)^k
{\textstyle \bigg(\begin{matrix}
n \\
k\end{matrix}\bigg)}(n-2k)^{q-1}=0,
\end{equation}
holds for every $n> q\geq 2$ whenever $n+q$ is odd.
\end{remark}

We now start from formula \eqref{form}, but considering the integral on $(\varepsilon,\infty)$,
and only at the end we will take the limit $\varepsilon\to 0$. This allows us to move the integral
inside the sum. In what follows $\omega(\varepsilon)$ will
denote a generic function of $\varepsilon$, vanishing at zero as $\varepsilon\to 0$.
Moreover, for $\alpha\neq 0$, let us define
$${\mathsf E}_\varepsilon(\alpha)=\int_\varepsilon^\infty \frac{e^{i\alpha x}}{x}\,dx.
$$

\begin{lemma}
\label{Z1}
For every $q\geq 1$, every $\varepsilon>0$ and every $\alpha\neq 0$, we have
$$\int_{\varepsilon}^\infty\frac{e^{i \alpha x}-P_{q-2}(i\alpha x)}{x^q}\, dx
=c_q\alpha^{q-1}+\frac{(i\alpha)^{q-1}}{(q-1)!}{\mathsf E}_\varepsilon(\alpha)+\omega(\varepsilon),
$$
where
$c_q=\frac{i^{q-1}}{(q-1)!}\sum_{k=0}^{q-2}\frac{1}{k+1}$ for $q\geq 2$ and $c_{1}=0$.
\end{lemma}

\begin{proof}
The proof goes by induction on $q$.
If $q=1$, equality holds with $\omega(\varepsilon)=0$.
Then, we prove the formula for $q+1$,
assuming it true for $q\geq 1$.
Since
$P'_{q-1}=P_{q-2}$,
an integration by parts yields
$$\int_{\varepsilon}^\infty\frac{e^{i\alpha x}-P_{q-1}(i\alpha x)}{x^{q+1}}\, dx
=\frac{e^{i\alpha\varepsilon}-P_{q-1}(i\alpha\varepsilon)}{q\varepsilon^q}
+\frac{i\alpha}{q}\int_{\varepsilon}^\infty\frac{e^{i\alpha x}-P_{q-2}(i\alpha x)}{x^{q}}\, dx.
$$
By the inductive hypothesis,
$$\frac{i\alpha}{q}\int_{\varepsilon}^\infty\frac{e^{i\alpha x}-P_{q-2}(i\alpha x)}{x^q}\, dx
=\frac{i c_q}{q}\alpha^q+\frac{(i\alpha)^{q}}{q!}{\mathsf E}_\varepsilon(\alpha)+\omega_q(\varepsilon),
$$
for some function $\omega_q$ vanishing at zero.
Noting that
$$\varpi_q(\varepsilon)=-\frac{(i \alpha)^q}{q!\, q}+\frac{e^{i\alpha \varepsilon}-P_{q-1}(i \alpha\varepsilon)}{q\varepsilon^q}
\to 0\quad\text{as }\varepsilon\to 0,
$$
we end up with the equality
$$\int_{\varepsilon}^\infty\frac{e^{i\alpha x}-P_{q-1}(i\alpha x)}{x^{q+1}}\, dx
=\bigg[\frac{i^q}{q!\, q}+\frac{ic_q}{q}\bigg]\alpha^q
+\frac{(i\alpha)^{q}}{q!}{\mathsf E}_\varepsilon(\alpha)+\omega_{q}(\varepsilon)+\varpi_{q}(\varepsilon).
$$
The final observation that
$\frac{i^q}{q!\, q}+\frac{ic_q}{q}=c_{q+1}$
completes the proof.
\end{proof}

\begin{proof}[Proof of Theorem \ref{UNO} for the case $n+q$ even]
Substituting the expression given by Lemma~\ref{Z1} into \eqref{form},
and
noting that
$${\mathsf E}_\varepsilon(n-2k)-{\mathsf E}_\varepsilon(-(n-2k))=2i{\mathsf {Si}}((n-2k)\varepsilon),
$$
where
$${\mathsf {Si}}(t)=\int_t^\infty \frac{\sin x}{x}\,dx$$
is the \emph{SinIntegral} function, we obtain
$$I_{n,q}
={\displaystyle \frac{(-1)^{\frac{q-n}2}}{2^{n-1}(q-1)!}\sum_{k=0}^{\lfloor\frac{n-1}2\rfloor}}(-1)^k
{\textstyle \bigg(\begin{matrix}
n \\
k\end{matrix}\bigg)}(n-2k)^{q-1}{\mathsf {Si}}((n-2k)\varepsilon)+\omega(\varepsilon).
$$
Since
$${\mathsf {Si}}((n-2k)\varepsilon)\to {\mathsf {Si}}(0)=\frac\pi2\quad\text{as }\varepsilon\to 0,$$
the result follows.
\end{proof}

\begin{proof}[Proof of Theorem \ref{UNO} for the case $n+q$ odd]
Again, we substitute the expression given by Lemma~\ref{Z1} into \eqref{form}.
Using \eqref{shortino}, and
noting that
$${\mathsf E}_\varepsilon(n-2k)+{\mathsf E}_\varepsilon(-(n-2k))=2{\mathsf {Ci}}((n-2k)\varepsilon),
$$
where
$${\mathsf {Ci}}(t)=\int_t^\infty \frac{\cos x}{x}\,dx$$
is the \emph{CosIntegral} function, we obtain
$$I_{n,q}
={\displaystyle \frac{(-1)^{\frac{q-n-1}2}}{2^{n-1}(q-1)!}\sum_{k=0}^{\lfloor\frac{n-1}2\rfloor}}(-1)^k
{\textstyle \bigg(\begin{matrix}
n \\
k\end{matrix}\bigg)}(n-2k)^{q-1}{\mathsf {Ci}}((n-2k)\varepsilon)+\omega(\varepsilon).
$$
By a further use of \eqref{shortino}, we can replace ${\mathsf {Ci}}((n-2k)\varepsilon)$ with
$${\mathsf {Ci}}((n-2k)\varepsilon)-{\mathsf {Ci}}(\varepsilon)\to -\log(n-2k)\quad\text{as }\varepsilon\to 0,$$
and a final limit $\varepsilon\to 0$ completes the argument.
\end{proof}

\bigskip

\end{document}